\documentclass[11pt]{amsart}

\usepackage{amssymb,amsthm,amsmath}
\usepackage{color}
\definecolor{darkblue}{rgb}{0, 0, .4}
\definecolor{grey}{rgb}{.7, .7, .7}
\usepackage[breaklinks]{hyperref}
\hypersetup{
    colorlinks=true,
    linkcolor=darkblue,
    anchorcolor=darkblue,
    citecolor=darkblue,
    pagecolor=darkblue,
    urlcolor=darkblue,
    pdftitle={},
    pdfauthor={}
}

\newtheorem{theorem}{Theorem}[section]
\newtheorem{lemma}[theorem]{Lemma}

\theoremstyle{definition}
\newtheorem{definition}[theorem]{Definition}
\newtheorem{example}[theorem]{Example}

\theoremstyle{remark}
\newtheorem{remark}[theorem]{Remark}

\numberwithin{equation}{section}

\theoremstyle{theorem}
\newtheorem{corollary}[theorem]{Corollary}

\newtheorem{proposition}[theorem]{Proposition}

\newcommand{\N}[0]{\mathbb{N}}

\newcommand{\ul}[1]{\underline{#1}}

\newcommand{\union}{\bigcup}

\newcommand{\hs}{}  
\newcommand{\hf}{\bullet}  
\newcommand{\hb}{{\color{grey} \bullet}}  

\newcommand{\heap}{ \xymatrix @=-9pt @! } 
\def\SDSize{6}  
\def\SDSizeTANRIGHT{4}  
\def\SDSizeTANLEFT{2}  
\def\SDMidpt{3}  
\def\SDColor{blue}
\def\SDEColor{black}
\newcommand{\StringRXD}[1]{{\color{\SDColor}\xy (\SDSize, \SDSize)*{}; (0, 0)*{}; **\crv{~**\dir{.}(\SDMidpt,\SDMidpt)};  (\SDMidpt, \SDMidpt)*{{\color{\SDEColor}#1}}; \endxy }}
\newcommand{\StringLXD}[1]{{\color{\SDColor}\xy (0, \SDSize)*{}; (\SDSize, 0)*{}; **\crv{~**\dir{.}(\SDMidpt,\SDMidpt)};  (\SDMidpt, \SDMidpt)*{{\color{\SDEColor}#1}}; \endxy }}
\newcommand{\StringL}[1]{{\color{\SDColor}\xy (\SDSize, \SDSize)*{}; (0, 0)*{}; (0, \SDSize)*{}; **\crv{(\SDSizeTANLEFT,\SDMidpt)};  (\SDMidpt, \SDMidpt)*{{\color{\SDEColor}#1}}; \endxy}}
\newcommand{\StringR}[1]{{\color{\SDColor}\xy (0,0)*{}; (\SDSize, 0)*{}; (\SDSize, \SDSize)*{}; **\crv{(\SDSizeTANRIGHT,\SDMidpt)}; (\SDMidpt, \SDMidpt)*{{\color{\SDEColor}#1}}; \endxy }}

\newcommand{\StringRX}[1]{{\color{\SDColor}\xy (\SDSize, \SDSize)*{}; (0, 0)*{}; **\crv{(\SDMidpt,\SDMidpt)};  (\SDMidpt, \SDMidpt)*{{\color{\SDEColor}#1}}; \endxy }}
\newcommand{\StringLR}[1]{{\color{\SDColor}\xy (0, 0)*{}; (0, \SDSize)*{}; **\crv{(\SDSizeTANLEFT,\SDMidpt)};  (\SDSize, 0)*{}; (\SDSize, \SDSize)*{}; **\crv{( \SDSizeTANRIGHT,\SDMidpt)}; (\SDMidpt, \SDMidpt)*{{\color{\SDEColor}#1}}; \endxy }}
\newcommand{\StringLRX}[1]{{\color{\SDColor}\xy (0, \SDSize)*{}; (\SDSize, 0)*{}; **\crv{(\SDMidpt,\SDMidpt)};  (\SDSize, \SDSize)*{}; (0, 0)*{}; **\crv{(\SDMidpt,\SDMidpt)}; (\SDMidpt, \SDMidpt)*{{\color{\SDEColor}#1}}; \endxy }}

\usepackage[all, dvips, knot]{xy}
\usepackage[all, knot, dvips, color]{xypic}
\xyoption{arc}

\newcommand{\h}{\mathsf{h}}
\newcommand{\w}{\mathsf{w}}

\begin{document}

\title{The enumeration of maximally clustered permutations}

\begin{abstract}
The maximally clustered permutations are characterized by avoiding the
classical permutation patterns $\{3421, 4312, 4321 \}$.  This class contains
the freely braided permutations and the fully commutative permutations.  In
this work, we show that the generating functions for certain fully commutative
pattern classes can be transformed to give generating functions for the
corresponding freely braided and maximally clustered pattern classes.
Moreover, this transformation of generating functions is rational.  As a
result, we obtain enumerative formulas for the pattern classes mentioned above
as well as the corresponding hexagon-avoiding pattern classes where the
hexagon-avoiding permutations are characterized by avoiding $\{46718235,
46781235, 56718234, 56781234\}$.
\end{abstract}

\author{Hugh Denoncourt}
\address{Department of Mathematics, Box 395, Boulder, Colorado 80309-0395}
\email{\href{mailto:hugh.denoncourt@colorado.edu}{\texttt{hugh.denoncourt@colorado.edu}}}
\urladdr{\url{http://math.colorado.edu/\~denoncou/}}

\author{Brant C. Jones}
\address{Department of Mathematics, One Shields Avenue, University of California, Davis, CA 95616}
\email{\href{mailto:brant@math.ucdavis.edu}{\texttt{brant@math.ucdavis.edu}}}
\urladdr{\url{http://www.math.ucdavis.edu/\~brant/}}

\thanks{The second author received support from NSF grant DMS-9983797.}


\date{\today}

\maketitle



\bigskip
\section{Introduction}\label{s:background}

The maximally clustered permutations introduced in \cite{losonczy} are a
generalization of the freely braided permutations developed in \cite{g-l1} and
\cite{g-l2}, and these in turn include the fully commutative permutations
studied in \cite{s1} as a subset.  In \cite{j2}, an explicit formula was
obtained for the Kazhdan--Lusztig polynomials of maximally-clustered
hexagon-avoiding permutations, generalizing an earlier result of \cite{b-w}
that identified the 321-hexagon avoiding permutations.  The enumeration of the
321-hexagon avoiding permutations was first given by \cite{s-w} who showed
that these elements satisfy a linear constant-coefficient recurrence with 7
terms.

\begin{theorem}{\bf \cite{s-w}}\label{t:s-w}
The number $c_n$ of 321-hexagon-avoiding permutations in $S_n$ satisfies the
recurrence
\[ c_{n+1} = 6 c_{n} - 11c_{n-1} + 9c_{n-2} - 4c_{n-3} - 4c_{n-4} + c_{n-5} \]
for all $n \geq 8$ with initial conditions given in Figure~\ref{f:mc.enum}.
\end{theorem}

Theorem~\ref{t:s-w} was extended in \cite{mansour-stankova} and also proved
using an enumeration scheme as described in \cite{vatter}.
Figure~\ref{f:mc.enum} shows the number $b_n$ and $m_n$ of freely-braided
hexagon-avoiding and maximally-clustered hexagon-avoiding permutations
respectively, in $S_n$ for $n \leq 15$.  It can be shown that each of these
pattern classes have a rational generating function by Proposition 13 and
Corollary 10 of \cite{a-l-r}.  
In this paper, we develop a method for determining the generating function
precisely.  

{\tiny
\begin{figure}[h]
\begin{center}
\begin{tabular}{|p{1.05in}|p{0.1in}|p{0.15in}|p{0.15in}|p{.15in}|p{0.15in}|p{0.2in}|p{0.2in}|p{0.25in}|p{0.25in}|p{0.3in}|p{.35in}|p{0.4in}|p{0.4in}|p{0.4in}|}
\hline
$n$ & 3 & 4 & 5 & 6 & 7 & 8 & 9 & 10 & 11 & 12 & 13 & 14 & 15 \\
\hline
\# $[321]$-hexagon avoiding & 5 & 14 & 42 & 132 & 429 & 1426 & 4806 & 16329 & 55740 & 190787 & 654044 & 2244153 & 7704047 \\
\hline
\# freely-braided hexagon-avoiding & 6 & 20 & 71 & 260 & 971 & 3670 & 13968 & 53369  & 204352  & 783408  & 3005284 &  11533014 &  44267854 \\
\hline
\# maximally-clustered hexagon-avoiding & 6 & 21 & 78 & 298 & 1157 & 4535 & 17872 & 70644 & 279704  & 1108462 & 4395045 & 17431206 &  69144643 \\
\hline
\end{tabular}
\end{center}
\caption{Enumeration of hexagon-avoiding classes}
\label{f:mc.enum}
\end{figure}
}

One feature of our main results, Theorem~\ref{t:enumerate} and
Theorem~\ref{t:main_diamond} below, is that they provide many examples of
classes characterized by permutation pattern avoidance that have rational
generating functions.  We also apply Theorem~\ref{t:enumerate} to enumerate the
maximally-clustered permutations introduced in \cite{losonczy}.  The freely
braided permutations have previously been enumerated in \cite{mansour} and we
recover this result from Theorem~\ref{t:enumerate} as well.

Section~\ref{s:background} describes the combinatorial model used in our
enumeration.  In Section~\ref{s:enumeration} we give a result that enables us
to find enumerative formulas for the pattern classes mentioned above in a
unified way.  In Section~\ref{s:diamond} we show how the methods of proof from
Section~\ref{s:enumeration} can be applied in the fully commutative case to
recover the main result of \cite{s-w} as well as some new generating functions.

\subsection{Background}

We view the symmetric group $S_n$ as the Coxeter group of type $A$ with
generating set $S = \{ s_1, \dots, s_{n-1} \}$ and relations of the form $(s_i
s_{i \pm 1})^3 = 1$ together with $(s_{i}s_{j})^{2} = 1$ for $| i - j | \geq
2$ and ${s_i}^2 = 1$.  We also refer to elements in the symmetric group by the \textit{1-line
notation} $w=[w_{1}w_{2} \dots w_{n}]$ where $w$ is the bijection mapping $i$
to $w_{i}$.  Then the generators $s_{i}$ are the adjacent transpositions
interchanging the entries $i$ and $i+1$ in the 1-line notation.  Suppose
$w=[w_1 \dots w_n]$ and $p=[p_1 \ldots p_k]$ is another permutation in $S_{k}$
for $k\leq n$.  We say $w$ \textit{contains the permutation pattern} $p$ or $w$
\textit{contains} $p$ \textit{as a 1-line pattern} whenever there exists a
subsequence $1\leq i_{1}<i_{2}<\ldots<i_{k}\leq n$ such that
\begin{equation*}
w_{i_{a}} < w_{i_{b}} \text{ if and only if } p_{a} < p_{b}
\end{equation*}
for all $1 \leq a < b \leq k$.  We call $(i_1, i_2, \dots, i_k)$ the \em
pattern instance\em.  For example, $[\ul{5}32\ul{4}\ul{1}]$ contains the
pattern $[321]$ in several ways, including the underlined subsequence.  If $w$
does not contain the pattern $p$, we say that $w$ \em avoids \em $p$.
A \em pattern class \em is a set of permutations characterized by avoiding a set
of permutation patterns.  For example, the maximally clustered permutations are
characterized by avoiding the permutation patterns
\begin{align}\label{e:mc.patterns}
    \text{ $[3421]$, $[4312]$, and $[4321]$ }
\end{align}
by \cite[Proposition 3.7]{losonczy},
while the freely braided permutations are characterized by avoiding
\begin{align}\label{e:fb.patterns}
    \text{ $[4231]$, $[3421]$, $[4312]$, and $[4321]$ }
\end{align}
as permutation patterns by \cite[Proposition 5.1.1]{g-l1}.

Recall that the products of generators from $S$ with a minimal number of
factors are called \em reduced expressions\em, and $l(w)$ is the length of such
an expression for $w \in S_n$.  Given $w \in S_n$, we represent reduced
expressions for $w$ in sans serif font, say $\w=\w_{1} \w_{2}\cdots \w_{p}$
where each $\w_{i} \in S$.  We call any expression of the form $s_i s_{i \pm 1}
s_i$ a \em short-braid\em.  There is a well-known theorem of Matsumoto
\cite{matsumoto} and Tits \cite{t}, which states that any reduced expression for
$w$ can be transformed into any other by applying a sequence of relations of the
form $(s_i s_{i \pm 1})^3 = 1$ together with $(s_{i}s_{j})^{2} = 1$ for $| i - j
| > 1$.  The theorem implies that the set of all generators appearing in any
reduced expression for $w$ is well-defined.  We call this set of generators the
\em support \em of $w$ and denote it by $supp(w)$.  We say that the permutation
$w$ is \em connected \em if the subscripts of the generators appearing in
$supp(w)$ form a nonempty interval in $\{1, 2, \ldots, n-1\}$.

As in \cite{s1}, we define an equivalence relation on the set of reduced
expressions for a permutation by saying that two reduced expressions are in the
same \em commutativity class \em if one can be obtained from the other by a
sequence of \em commuting moves \em of the form $s_i s_j \mapsto s_j s_i$ where
$|i-j| \geq 2$.  If the reduced expressions for a permutation $w$ form a single
commutativity class, then we say $w$ is \em fully commutative\em.

\subsection{Heaps}

If $\w = \w_1 \cdots \w_k$ is a reduced expression, then as in \cite{s1} we
define a partial ordering on the indices $\{1, \cdots, k\}$ by the transitive
closure of the relation $i \prec j$ if $i < j$ and $\w_i$ does not commute with
$\w_j$.  We label each element $i$ of the poset by the corresponding generator
$\w_i$.  It follows from the definition that if $\w$ and $\w'$ are two reduced
expressions for a permutation $w$ that are in the same commutativity class, then
the labeled posets of $\w$ and $\w'$ are isomorphic.  This isomorphism class of
labeled posets is called the \em heap \em of $\w$, where $\w$ is a reduced
expression representative for a commutativity class of $w$.  In particular, if
$w$ is fully commutative then it has a single commutativity class, and so there
is a unique heap of $w$.

As in \cite{b-w}, we will represent a heap as a set of lattice points embedded
in $\N^2$.  To do this, we assign coordinates $(x,y) \in \N^2$ to each entry of
the labeled Hasse diagram for the heap of $\w$ in such a way that:
\begin{enumerate}
\item[(1)] An entry represented by $(x,y)$ is labeled $s_i$ in the heap if and
only if $x = i$, and
\item[(2)] If an entry represented by $(x,y)$ is greater than an entry
represented by $(x',y')$ in the heap, then $y > y'$.
\end{enumerate}
Since the Coxeter graph of type $A$ is a path, it follows from the definition
that $(x,y)$ covers $(x',y')$ in the heap if and only if $x = x' \pm 1$, $y >
y'$, and there are no entries $(x'', y'')$ such that $x'' \in \{x, x'\}$ and $y'
< y'' < y$.  Hence, we can completely reconstruct the edges of the Hasse
diagram and the corresponding heap poset from a lattice point representation.
This representation will enable us to make arguments ``by picture'' that would
otherwise be difficult to formulate.  Although there are many coordinate
assignments for any particular heap, the $x$ coordinates of each entry are
fixed for all of them, and the coordinate assignments of any two entries only
differs in the amount of vertical space between them.

\begin{example}\label{ex:heap}
One lattice point representation of the heap of $w = s_2 s_3 s_1 s_2 s_4$ is
shown below, together with the labeled Hasse diagram for the unique heap poset
of $w$.

\smallskip
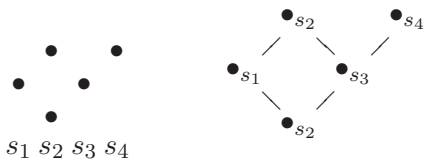
\begin{center}
\begin{tabular}{ll}
\xymatrix @=-4pt @! {
\hs & \hs & \hs & \hs & \hs \\
& \hf & \hs & \hf & \hs & \hs \\
\hf & \hs & \hf & \hs & \hs & \hs \\
& \hf & \hs & \hs & \hs & \hs \\
s_1 & s_2 & s_3 & s_4 \\
} &
\xymatrix @=0pt @! {
 & \hf_{s_2} \ar@{-}[dl] \ar@{-}[dr] & & \hf_{s_4} \ar@{-}[dl] \\
 \hf_{s_1} \ar@{-}[dr] & & \hf_{s_3} \ar@{-}[dl] \\
 & \hf_{s_2} \\
} \\
\end{tabular}
\end{center}
\end{example}

Suppose $x$ and $y$ are a pair of entries in the heap of $\w$ that correspond
to the same generator $s_i$, so that they lie in the same column $i$ of the heap.
Assume that $x$ and $y$ are a \em minimal pair \em in the sense that there is
no other entry between them in column $i$.  Then, for $\w$ to be reduced, there
must exist at least one non-commuting generator between $x$ and $y$, and if
$\w$ is short-braid avoiding, there must actually be two non-commuting labeled
heap entries that lie strictly between $x$ and $y$ in the heap.  We call these
two non-commuting labeled heap entries a \em resolution \em of the pair $x,y$.
If the generators lie in distinct columns, we call the resolution a \em distinct
resolution\em.  The Lateral Convexity Lemma of \cite{b-w} characterizes
fully commutative permutations $w$ as those for which every minimal pair in the
heap of $w$ has a distinct resolution.

\begin{definition} \label{d:supports}
If $s_i \in supp(w)$, we say that $s_i$ \emph{supports} $w$.  We also
say that column $i$ supports $w$.  If $w$ is fully commutative and 
connected then $supp(w) = \{s_i, s_{i+1}, \ldots, s_j\}$
for some $i, j$, and since every minimal pair of entries in the heap of $w$ has
a distinct resolution we must have exactly one entry in columns $i$ and $j$ of
the heap.  In this situation, we call columns $i+1, i+2, \ldots, j-1$ the \em
internal columns \em of the heap of $w$ and we call columns $i$ and $j$ the \em
extremal columns \em of the heap of $w$.
\end{definition}

We now describe a notion of containment for heaps.  Recall from \cite{b-j1} that
an \em orientation preserving Coxeter embedding \em $f: \{s_1, \dots, s_{k-1} \}
\rightarrow \{s_1, \dots, s_{n-1} \}$ is an injective map of Coxeter generators such
that for each $m \in \{2, 3\}$, we have
\[ (s_i s_j)^{m} = 1 \text{ if and only if } (f(s_i) f(s_j))^{m} = 1 \]
and the subscript of $f(s_i)$ is less than the subscript of $f(s_j)$ whenever $i
< j$.  We view this as a map of permutations, which we also denote $f: S_k
\rightarrow S_n$, by extending it to a word homomorphism which can then be
applied to any reduced expression in $S_k$.

Recall that a subposet $Q$ of $P$ is called \em convex \em if $y \in Q$
whenever $x < y < z$ in $P$ and $x, z \in Q$.  Suppose that $w$ and $h$ are
permutations.  We say that $w$ \em heap-contains \em $h$ if there exist
commutativity classes represented by $\w$ and $\h$, together with an orientation
preserving Coxeter embedding $f$ such that the heap of $f(\h)$ is contained as
a convex labeled subposet of the heap of $\w$.  If $w$ does not heap-contain
$h$, we say that $w$ \em heap-avoids \em $h$.  To illustrate, $w = s_2 s_3 s_1
s_2 s_4$ from Example~\ref{ex:heap} heap-contains $s_1 s_2 s_3$ under the
Coxeter embedding that sends $s_i \mapsto s_{i+1}$, but $w$ heap-avoids $s_1
s_2 s_1$.

In type $A$, the heap construction can be combined with another
combinatorial model for permutations in which the entries from the
1-line notation are represented by strings.  The points at which two
strings cross can be viewed as adjacent transpositions of the 1-line
notation.  Hence, we may overlay strings on top of a heap diagram to
recover the 1-line notation for the permutation, by drawing the strings
from bottom to top so that they cross at each entry in the
heap where they meet and bounce at each lattice point not in
the heap.  Conversely, each permutation string diagram corresponds
with a heap by taking all of the points where the strings cross as the
entries of the heap.

For example, we can overlay strings on the two heaps of $[3214]$.
Note that the labels in the picture below refer to the strings, not
the generators.
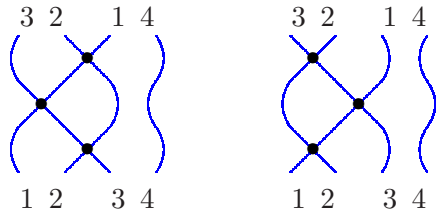
\begin{center}
\begin{tabular}{ll}
    \heap{
    & 3 \ \ 2 &  & 1 \ \ 4 & \\
    \StringR{\hs} & \hs & \StringLRX{\hf} & \hs & \StringL{\hs} \\
    \hs & \StringLRX{\hf} & \hs & \StringLR{\hs} & \hs \\
    \StringR{\hs} & \hs & \StringLRX{\hf} & \hs & \StringL{\hs} \\
    & 1 \ \ 2 &  & 3 \ \ 4 & \\
    } &
    \heap{
    & 3 \ \ 2 &  & 1 \ \ 4 &  \\
    \hs & \StringLRX{\hf} & \hs & \StringLR{\hs} \\
    \StringR{\hs} & \hs & \StringLRX{\hf} & \hs & \StringL{\hs} \\
    \hs & \StringLRX{\hf} & \hs & \StringLR{\hs} \\
    & 1 \ \ 2 &  & 3 \ \ 4 &  \\
    } \\
\end{tabular}
\end{center}

For a more leisurely introduction to heaps and string diagrams, as well as
generalizations to Coxeter types $B$ and $D$, see \cite{b-j1}.  Cartier and
Foata \cite{cartier-foata} were among the first to study heaps of dimers, which
were generalized to other settings by Viennot \cite{viennot}.  Stembridge has
studied enumerative aspects of heaps \cite{s1,s2} in the context of fully
commutative elements.  Green has also considered heaps of pieces with
applications to Coxeter groups in \cite{green1,green2,green3}.

\subsection{Maximally clustered elements}

In \cite{losonczy}, Losonczy introduced the maximally clustered elements of
simply laced Coxeter groups.

\begin{definition}{\bf \cite{losonczy}}\label{d:mc}
A \em braid cluster \em is an expression of the form \[ s_{i_1} s_{i_2} \dots
s_{i_k} s_{i_{k+1}} s_{i_k} \dots s_{i_2} s_{i_1} \] where each $s_{i_p}$ for
$1 \leq p \leq k$ has a unique $s_{i_q}$ with $p < q \leq k+1$ such that
$|i_p - i_q| = 1$.

Let $w$ be a permutation and let $N(w)$ denote the number of $[321]$ pattern
instances in $w$.  We say $w$ is \em maximally clustered \em if there is a
reduced expression for $w$ of the form
\[ a_0 c_1 a_1 c_2 a_2 \dots c_M a_M \]
where each $a_i$ is a reduced expression, each $c_i$ is a braid cluster with
length $2 n_i + 1$ and $N(w) = \sum_{i=1}^{M} n_i$.  Such an expression is
called \em contracted\em.  In particular, $w$ is \em freely braided \em if
there is a reduced expression for $w$ with $N(w)$ disjoint short-braids.
\end{definition}

This is not the original definition for the maximally clustered elements;
however it is equivalent.  The remarks in Section 5 of \cite{g-l1} show that the
number of $[321]$ pattern instances in $w$ equals the number of contractible
triples of roots in the inversion set of $w$.  Corollary 4.11(ii) and
Corollary 4.13 of \cite{losonczy} prove that $\w$ is a contracted reduced
expression for a maximally clustered element if and only if it has the form
given in Definition~\ref{d:mc}.  Moreover, it follows from the proof of
\cite[Corollary 4.11(ii)]{losonczy} that the $a_i$ in Definition~\ref{d:mc} are
fully commutative.

Recall that \cite{b-j-s} showed that $w$ is fully commutative whenever $N(w) =
0$.  In this work we will frequently use the fact that any braid cluster has the
canonical form of Lemma~\ref{l:braid_cluster}.

\begin{lemma}\label{l:braid_cluster}
Suppose $x = s_{i_1} s_{i_2} \dots s_{i_{k}} s_{i_{k+1}} s_{i_{k}} \dots s_{i_2}
s_{i_1}$ is a braid cluster of length $2 k + 1$ in type $A$.  Then, $x = s_{m+1}
s_{m+2} \dots s_{m+{k}} s_{m+k+1} s_{m+{k}} \dots s_{m+2} s_{m+1}$ for some $m$.
\end{lemma}
\begin{proof}
Since $s_{i_{k+1}}$ is a transposition and conjugation preserves cycle type, $x$
is a transposition and we may write $x = (m+1 \,\,\,\, m+k'+2)$ in cycle
notation for some $k',m \geq 0$.  This transposition is given by the expression
$\mathsf{x} = s_{m+1} s_{m+2} \cdots s_{m+k'} s_{m+k'+1} s_{m+k'} \cdots s_{m+2}
s_{m+1}$, which is a reduced expression for $x$ by \cite[Lemma 4.3]{losonczy}.
Since the length of reduced expressions for $x$ is an invariant of $x$, $k' = k$
and the result follows.
\end{proof}

Recall the following structural lemma about contracted reduced expressions.

\begin{lemma}{\bf \cite[Lemma 2.3]{j2}}\label{l:unique}
Let $\w$ be a contracted reduced expression for a maximally clustered
permutation, so $\w$ has the form
\[ a_0 c_1 a_1 \dots c_{M} a_{M} \]
where each $c_j$ is a braid cluster, and the $a_j$ are short-braid
avoiding.  Then, any generator $s_i$ that appears in any of the braid clusters
$c_j$ does not appear anywhere else in $\w$.
\end{lemma}

\begin{lemma}\label{l:constant_support}
Let $\w$ be a contracted reduced expression for a maximally clustered
permutation.  If the generator $s_i$ supports a braid cluster in $\w$ then $s_i$
supports a braid cluster in every contracted reduced expression for $w$.
\end{lemma}
\begin{proof}
By \cite[Lemma 4.11(i)]{losonczy}, the contracted reduced expressions form a
complete set of representatives for the commutativity classes of $w$.  Suppose
the generator $s_i$ supports a braid cluster in $\w$ in the sense of
Definition~\ref{d:supports} but there exists a contracted reduced expression
$\w' = a_0 c_1 a_1 \ldots c_M a_M$ for $w$ in which $s_i$ does not support a
braid cluster.  By the theorem of Matsumoto \cite{matsumoto} and Tits \cite{t},
it suffices to consider a pair of heaps represented by $\w$ and $\w'$ that are
related by a single short-braid move.

Observe that each short-braid move on $\w'$ can change the length of at most
one braid cluster.  This is clear if the short-braid move involves two entries
from a single braid cluster, or one entry from some braid cluster $c_i$
together with some entry of $a_j$ for $j \in \{i-1, i\}$.  In the case that
some $a_i$ is the identity, note that there are no short-braid moves involving
an entry from $c_i$ and an entry from $c_{i+1}$ by Lemma~\ref{l:unique}.

Therefore, by the equation $N(w) = \sum_{i = 1}^{M} n_i$ from
Definition~\ref{d:mc}, we actually have that no short-braid move on $\w'$ can
change the length of any of the braid clusters because $N(w)$ does not depend on
the reduced expression for $w$.  Thus we have shown that the length of each
braid cluster remains the same over all contracted reduced expressions for $w$.
Finally, since no single short-braid move can change the support of a
braid-cluster without changing its length, we have shown that the support of
each braid cluster also remains the same over all contracted reduced expressions
for $w$.
\end{proof}

Putting these lemmas together, we can show that there is a canonical heap
associated to any maximally clustered permutation.  See
Example~\ref{ex:canonical} for an illustration.

\begin{definition}\label{d:bc_column_decomposition}
Let $\mathfrak{C}$ be a commutativity class for a permutation $w$.
Suppose that the set of columns $\{1, \ldots, n\}$ of the heap of $\mathfrak{C}$ can be
partitioned into intervals
\[ \tilde{C}_0 = [1, p_1 -1], \tilde{B}_1 = [p_1, q_1], \tilde{C}_1 = [q_1+1,
p_2-1], \ldots, \tilde{B}_k = [p_k, q_k], \tilde{C}_k = [q_k + 1, n] \]
satisfying:
\begin{enumerate}
\item[(1)]  For each $i$, the heap of $\mathfrak{C}$ restricted to
$\tilde{B}_i$ is a braid cluster in the canonical form of
Lemma~\ref{l:braid_cluster}.
\item[(2)]  For each $i$, every minimal pair of entries in any column from
$\tilde{C}_i$ has a distinct resolution, possibly using entries from columns
$q_i$ or $p_{i+1}$.
\end{enumerate}
Then, we say that the heap of $\mathfrak{C}$ has a \em braid cluster column
decomposition \em given by $\tilde{C}_0, \tilde{B}_1, \tilde{C}_1, \ldots
\tilde{B}_k, \tilde{C}_k$.
\end{definition}

\begin{proposition}\label{p:canonical_form}
Every maximally clustered permutation $w$ has a unique commutativity class
$\mathfrak{C}$ such that the heap of $\mathfrak{C}$ has a braid cluster
column decomposition.  Conversely, every heap having a braid cluster column
decomposition corresponds to a maximally clustered permutation.
\end{proposition}
\begin{proof}
Let $w$ be a maximally clustered permutation.  By \cite[Lemma
4.11(i)]{losonczy}, the contracted reduced expressions form a complete set of
representatives for the commutativity classes of $w$.  By
Lemma~\ref{l:constant_support}, each generator either supports a unique braid
cluster in some contracted reduced expression, in which case the generator must
also support the braid cluster in every contracted reduced expression, or the
generator is not used in any braid cluster.  Hence, the set of braid clusters
appearing in any contracted reduced expression for $w$ is independent of the
choice of reduced expression.  Therefore, specifying a commutativity class of
$w$ depends only on a choice of commutativity class for each braid cluster.  We
choose each braid cluster to have the form given in Lemma~\ref{l:braid_cluster}
and declare this to be our canonical commutativity class $\mathfrak{C}$ for $w$.

Consider the heap associated to $\mathfrak{C}$.  We can partition the support of
$w$ into segments $\tilde{B}_1, \tilde{B}_2, \ldots, \tilde{B}_k$ corresponding
to the generators that support braid clusters.  The complementary columns $\{1,
\ldots, n\} \setminus \bigcup_{i=1}^k \tilde{B}_i$ form intervals that we denote
$\tilde{C}_0, \tilde{C}_1, \ldots, \tilde{C}_k$.

By definition, this column partition satisfies property (1) of
Definition~\ref{d:bc_column_decomposition}.  Any minimal pair of entries from
some $\tilde{C}_i$ without a distinct resolution would either form a
short-braid or form an extension of an existing braid-cluster, contradicting
the definition of the column partition we have given.  Hence, the column
partition satisfies property (2) of
Definition~\ref{d:bc_column_decomposition} as well.

To prove the converse statement, fix a heap with a column partition satisfying
properties (1) and (2) and let $\w$ be a corresponding expression.  We show that
$\w$ is reduced and maximally clustered.

If $\w$ is not reduced then there exists a sequence of braid moves bringing $\w$
to an expression whose heap has two entries in the same column with no entries
between them.  But the available braid moves in any expression related to $\w$
only involve entries from some braid cluster supported on columns
$\tilde{B}_i$ because the entries in the complimentary columns $\union_{i = 0}^k
\tilde{C}_i$ all have distinct resolutions by property (2) and this remains true
after performing arbitrarily many braid moves in $\union_{i=1}^k \tilde{B}_i$.
Since the braid clusters supported on each $\tilde{B}_i$ are reduced, we never
have any opportunity for cancellation.  Hence, $\w$ is reduced.

Moreover, we can read a contracted reduced expression for $w$ from the given
heap.  We begin by linearly ordering the braid clusters
$\tilde{B}_1, \dots, \tilde{B}_k$ in a way which is compatible with the
order in which they appear in the heap poset.  First read all of the entries in
columns $\tilde{C} = \union_{i = 0}^k \tilde{C}_{i}$
that lie below the entries of $\tilde{B}_1$ in the heap poset.  Then for
each $i < k$ we sequentially read the entries of $\tilde{B}_i$ followed by
the entries from columns $\tilde{C}$ appearing between $\tilde{B}_i$ and
$\tilde{B}_{i+1}$ in the heap poset.  Finally, we read the entries from
columns $\tilde{C}$ that appear above $\tilde{B}_{k}$ in the heap poset.

\begin{figure}[h]
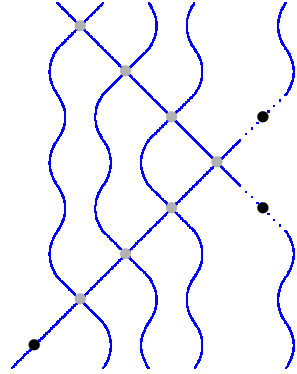

\begin{center}
\begin{tabular}{c}
\heap {
& & \StringLRX{\hb} & & \StringLR{\hs} & & \StringR{\hs}\\
& \StringR{\hs} & & \StringLRX{\hb} & & \StringL{\hs} & & \StringL{\hs}\\
& & \StringLR{\hs} & & \StringLRX{\hb} & & \StringRXD{\hf}\\
& \StringR{\hs} & & \StringLR{\hs} & & \StringLRX{\hb} & & {\hf} \\
& & \StringLR{\hs} & & \StringLRX{\hb} & & \StringLXD{\hf}\\
& \StringR{\hs} & & \StringLRX{\hb} & & \StringL{\hs} & & \StringL{\hs}\\
& & \StringLRX{\hb} & & \StringLR{\hs} & & \StringR{\hs}\\
& \StringRX{\hf} & & \StringLR{\hs} & & \StringL{\hs} & & \StringL{\hs}\\
} \\
\end{tabular}
\end{center}
\caption{Strings on a braid cluster}
\label{f:overlay}
\end{figure}

Since the braid clusters have the canonical form of Lemma~\ref{l:braid_cluster},
we can verify that each braid cluster $\tilde{B}_i$ supported on $n_i + 1$
columns has $n_i$ $[321]$-instances using a string diagram overlaid on the heap,
as illustrated in Figure~\ref{f:overlay}.  By Lemma~\ref{l:unique} each braid
cluster $\tilde{B}_{i}$ takes exactly one string lying next to a column in
$\tilde{C}_{i}$ and crosses it with exactly one string lying next to a column in
$\tilde{C}_{i+1}$.  Since each minimal pair of entries from any column of
$\tilde{C}_{i}$ has a distinct resolution by property (2), there are no other
$[321]$-instances in $w$.  Therefore, the number of $[321]$-instances in $w$ is
exactly $\sum_{i=0}^{k-1} n_i$, and each braid cluster $\tilde{B}_{i}$ has
length $2 n_i + 1$.  Hence, $w$ is maximally clustered by Definition~\ref{d:mc}.
\end{proof}

\begin{example}\label{ex:canonical}
Suppose $w$ is given by the contracted reduced expression
\[ (s_5) (s_1 s_2 s_3 s_4 s_3 s_2 s_1) (s_6 s_5 s_9) (s_7 s_8 s_7) (s_6).\]
Then the heap of $w$ is drawn below, together with its braid cluster column
decomposition.  The braid clusters are shown in grey.
{\scriptsize
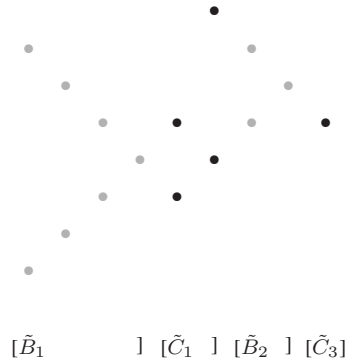
\[ \xymatrix @=-8pt @! {
& & {\hs} & & {\hs} & & {\hf} & & {\hs} & \\
& {\hb} & & {\hs} & & {\hs} & & {\hb} \\
& & {\hb} & & {\hs} & & {\hs} & & {\hb} & \\
& {\hs} & & {\hb} & & {\hf} & & {\hb} & & {\hf} \\
& & {\hs} & & {\hb} & & {\hf} & & {\hs} & \\
& {\hs} & & {\hb} & & {\hf} & & {\hs} & \\
& & {\hb} & & {\hs} & & {\hs} & & {\hs} \\
& {\hb} & & {\hs} & & {\hs} & & {\hs} & \\
& & {\hs} & & {\hs} & & {\hs} & & {\hs} \\
& \text{[} \tilde{B}_1 & & & \text{]} & \text{[} \tilde{C}_1 & \text{]} &
\text{[} \tilde{B}_2  & \text{]} & \text{[} \tilde{C}_3 \text{]} \\ } \] }
\end{example}


\bigskip
\section{Enumeration}\label{s:enumeration}

Let $S^{P} = \bigcup_{n \geq 1} S^{P}_{n}$ denote the permutations characterized
by avoiding a set of 1-line patterns $P$.  The most important pattern classes
for this work are the maximally clustered permutations and the freely braided
permutations, characterized by avoiding the patterns from \eqref{e:mc.patterns}
and \eqref{e:fb.patterns}, respectively.  Given a finite set $H$ of
permutations, let $S^{P}(H)$ be the subset of $S^{P}$ consisting of those
permutations that heap-avoid the patterns in $H$.

In \cite[Section 4]{j2} we described how to find a set of 1-line patterns $Q$
such that $S^{P}(\{ h \}) = S^{Q}$, when possible.  Example 11.1 of \cite{b-j1}
shows that some heap patterns $h$ have no such set $Q$.

Let $r(h)$ denote the rank of the symmetric group containing $h$.  Define
$U^{P}(h)$ to be the set of all elements in $S^{P}_{r(h)}$ that heap-contain
$h$.  This is a finite set because the rank is fixed.

\begin{definition}{\bf \cite[Definition 4.2]{j2}}\label{d:ideal}
Let $p \in S^{P}$.  Then, we say that $p$ is an \em ideal pattern in $S^{P}$ \em
if for every $q \in S^{P}_{r(p)+1}$ containing $p$ as a 1-line pattern, we have
that $q$ heap-contains $p$.
\end{definition}

For example, \cite[Theorem 3.8]{t2} implies that $p$ is an ideal pattern in
$S^{\emptyset}$ if $p$ avoids $[2143]$.  Definition~\ref{d:ideal} describes a
finite test that extends to permutations of all ranks according to the following
result.

\begin{theorem}{\bf \cite[Theorem 4.4]{j2}}\label{t:translate_olh}
Suppose $S^{P}(H)$ is the subset of permutations characterized by avoiding a
finite set $P$ of 1-line patterns and heap-avoiding a finite set $H$ of
permutations.  If each of the elements in $P' = \bigcup_{h \in H} U^{P}(h)$ is
an ideal pattern, then $S^{P}(H) = S^{P \union P'}$, so is characterized by
avoiding the permutations in $P \union P'$ as 1-line patterns.
\end{theorem}

We can apply Theorem~\ref{t:translate_olh} to study certain classical
permutation pattern classes using heap-avoidance.  For example, it is
straightforward to verify that
\[ S^{\{[321]\}}(\{ s_5 s_6 s_7 s_3 s_4 s_5 s_6 s_2 s_3 s_4 s_5 s_1 s_2 s_3 \}) =
S^{\{[321], [46718235], [46781235], [56718234], [56781234]\}} \]
using Theorem~\ref{t:translate_olh}.  Similar statements \cite[Corollary
4.5]{j2} hold for the freely braided and maximally clustered permutations.
The permutations that heap-avoid
\[ [46718235] = s_5 s_6 s_7 s_3 s_4 s_5 s_6 s_2 s_3 s_4 s_5 s_1 s_2 s_3 \]
are called \em hexagon-avoiding \em after \cite{b-w}.

Let $H$ be a finite set of connected fully commutative permutations each of
whose heaps contains at least two entries in each internal column in the sense
of Definition~\ref{d:supports} or let $H = \emptyset$.  Suppose $F_n$ is the
subset of the fully commutative permutations on $n$ generators that are
characterized by heap-avoiding the set of patterns from $H$, so $F_n =
S_{n+1}^{[321]}(H) \subset S_{n+1}^{[321]}$.  We define $|F_0|$ to be 1,
corresponding to the empty heap.  Our main result in this section is that we
can transform the generating function $F(x) = \sum_{n \geq 0} |F_n| \ x^n$ to
obtain generating functions for the corresponding freely braided and maximally
clustered pattern classes.  Moreover, the transformation is a rational function
of $F(x)$, so it preserves this important property of the generating function.

We first consider the number of permutations in $F_n$ that have one or both of
the extremal generators $\{ s_1, s_n \}$ present in the heap.  Recall that every
minimal pair of entries in a fully commutative heap must have a distinct
resolution, so in particular any fully commutative heap has at most a single
entry in its leftmost and rightmost columns.  Let $L_n$ be the permutations from
$F_n$ that have $s_n$ present in their heap, and $R_n$ be the set of
permutations from $F_n$ having $s_1$ present in the heap.  By the bijection that
reverses the ordering of the subscripts of the generators, we find that $|L_n| =
|R_n|$.  Let $M_n$ denote the set of permutations in $F_n$ with both $s_1$ and
$s_n$ present in the heap.  Then, we have the following enumerative result.

\begin{lemma}\label{l:pieces}
Let $H$ be a finite set of fully commutative permutations or let $H =
\emptyset$.  Suppose
\[ F(x) = \sum_{n \geq 0} |S_{n+1}^{[321]}(H)| \ x^n. \]
Then,
\[ L(x) = \sum_{n \geq 0} |L_n| \ x^n = F(x) - x F(x) - 1 \]
and
\[ M(x) = \sum_{n \geq 0} |M_n| \ x^n = F(x) - 2 x F(x) + x^2 F(x) - 1. \]
\end{lemma}
\begin{proof}
This follows using inclusion-exclusion together with the observation that $|F_0|
= 1$ and $|F_1| = 2$ so $|L_0| = |M_0| = |M_1| = 0$, which is required by
definition.
\end{proof}

We are now in a position to prove our main enumerative theorem.

\begin{theorem}\label{t:enumerate}
Let $H$ be a finite set of connected fully commutative permutations each of
whose heaps contains at least two entries in each internal column or let $H =
\emptyset$.  Suppose
\[ F(x) = \sum_{n \geq 0} |S_{n+1}^{[321]}(H)| \ x^n, \]
\[ L(x) = F(x) - x F(x) - 1, \text{ and } M(x) = F(x) - 2 x F(x) + x^2 F(x) - 1. \]
Then, we have
\[ \sum_{n \geq 0} |S_{n+1}^{\{[3421], [4231], [4312], [4321]\}}(H)| \ x^n =
F(x) + {{L(x)^2} \over {1 - M(x)}} \]
and
\[ \sum_{n \geq 0} |S_{n+1}^{\{[3421], [4312], [4321]\}}(H)| \ x^n = F(x) + {
{L(x)^2} \over {1 - x - M(x)}}. \]
\end{theorem}
\begin{proof}
Suppose $w$ is an element of $S_{n+1}^{\{[3421], [4312], [4321]\}}(H)$.  By
Proposition~\ref{p:canonical_form}, we may choose a commutativity class $\mathfrak{C}$ of
$w$ so that the
heap of $\mathfrak{C}$ has a specific form where each column either supports a braid
cluster, in which case there are no other generators in that column, or else
the column is not used in any braid cluster.

Let $k$ be the number of braid clusters in $w$.  We partition the columns $\{ 1,
\dots, n \}$ of the heap of $\mathfrak{C}$ into intervals $C_0, B_1, C_1, B_2, \dots, B_k,
C_k$ based on the location of the braid clusters in $w$.  We define the $r$th
interval $C_r$ to be $[p,q]$ where $p$ is the rightmost column of the $r$th
braid cluster and $q$ is the leftmost column of the $(r+1)$st braid cluster,
setting $p=1$ for $r=0$ and $q=n$ for $r=k$.  The $B_i$ then consist of the
internal columns of the $i$th braid cluster.  In particular, if $w$ is
freely braided then every $B_i = \emptyset$.  If $w$ is fully commutative, then
$C_0 = \{ 1, \dots, n \}$.

For example, suppose $w$ is given by the contracted reduced expression
\[ (s_5) (s_1 s_2 s_3 s_4 s_3 s_2 s_1) (s_6 s_5 s_9) (s_7
s_8 s_7) (s_6).\]
Then the heap of $w$ is drawn below, together with its column partition.  The
braid clusters are shown in grey.

{\scriptsize
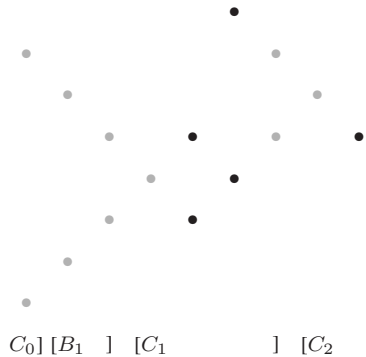
\[ \xymatrix @=-4pt @! {
& {\hs} & & {\hs} & & {\hf} & & {\hs} & \\
{\hb} & & {\hs} & & {\hs} & & {\hb} \\
& {\hb} & & {\hs} & & {\hs} & & {\hb} & \\
{\hs} & & {\hb} & & {\hf} & & {\hb} & & {\hf} \\
& {\hs} & & {\hb} & & {\hf} & & {\hs} & \\
{\hs} & & {\hb} & & {\hf} & & {\hs} & \\
& {\hb} & & {\hs} & & {\hs} & & {\hs} \\
{\hb} & & {\hs} & & {\hs} & & {\hs} & \\
C_0 \text{]} & \text{[} B_1 & \text{]} & \text{[} C_1 & & &
\text{]} & \text{[} C_2 \\
} \] }

Observe that the entries appearing in columns $B_i$ are completely determined by
the positions of the entries from the rightmost column in $C_{i-1}$ and the
leftmost column in $C_{i}$, as these are the ends of the braid cluster supported
by $B_i$.

Next, we define a map $\pi : \{C_0, C_1, \dots, C_k \} \rightarrow S_{n+1}$ to
project each interval of columns to a permutation.  If $k = 0$, then $\pi(C_0) =
C_0$.  Otherwise, observe that the leftmost column in each $C_i$ for $i > 0$ is
the rightmost column of some braid cluster, so it consists of a single entry.
The rightmost column $q$ in each $C_i$ for $i < k$ is the leftmost column of
some braid cluster, so it consists of two entries, and there can be no entries
between them in column $q-1$ by Definition~\ref{d:mc}.  Hence, we may define a
permutation $\pi(C_i)$, whose heap is obtained from the heap of $w$ restricted to
the columns of $C_i$, by collapsing the two entries in column $q$ to a single
entry.

In the example above, $\pi(C_1)$ has the heap
{\scriptsize 
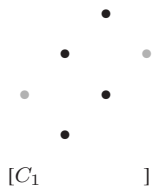
\[ \xymatrix @=-4pt @! {
{\hs} & & {\hf} & \\
& {\hf} & & {\hb} \\
{\hb} & & {\hf} & \\
& {\hf} & & {\hs} \\
\text{[} C_1 & & & \text{]} \\
} \] }

Next, we claim that for each $i$,
\begin{equation}\label{e:pieces}
\pi(C_i) \in
\begin{cases}
F_{|C_0|} & \text{ if $k = 0$, } \\
L_{|C_0|} & \text{ if $k > 0$ and $i = 0$, } \\
M_{|C_i|} & \text{ if $k > 0$ and $1 \leq i < k$, } \\
R_{|C_k|} & \text{ if $k > 0$ and $i = k$. }
\end{cases}
\end{equation}
In particular, we show that each $\pi(C_i)$ corresponds to a fully commutative
permutation that heap-avoids the patterns from $H$.

The claim is clear if $k = 0$, so suppose $k > 0$, and let $C_i = [p,q]$ with
$0 < i < k$.  Then $\pi(C_i)$ is fully commutative because every minimal pair
of entries in $\pi(C_i)$ has a distinct resolution by property (2) in
Definition~\ref{d:bc_column_decomposition} from
Proposition~\ref{p:canonical_form}.

Next, suppose there exists an instance of a heap-pattern $h \in H$ among the
entries of $\pi(C_i)$.  Then we must have such a heap-instance in $w$ which is a
contradiction.  This follows because the heap of $w$ with respect to the
commutativity class where the $(i+1)$st braid cluster is of the form
\[ s_{m+k+1} s_{m+k} \dots s_{m+1} s_{m} s_{m+1} \dots s_{m+k} s_{m+k+1} \]
has precisely the entries of $\pi(C_i)$ in columns $C_i$, and so $w$
heap-contains $h$.  As this cannot occur, each $\pi(C_i)$ heap-avoids the
patterns from $H$.

It follows from the construction that if $k > 0$ then $\pi(C_0)$ has an entry in
its rightmost column, $\pi(C_k)$ has an entry in its leftmost column, and every
other $\pi(C_i)$ has entries in both the leftmost and rightmost columns.
Otherwise, $k = 0$, and $\pi(C_0) = C_0 \in F_{|C_0|}$.  Thus, we have proved
(\ref{e:pieces}).

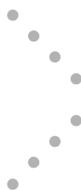
\begin{figure}[h]
\begin{center}
\begin{tabular}{c}
\xymatrix @=-4pt @! {
{\hb} & & {\hs} & & \\
& {\hb} & & {\hs} & & \\
{\hs} & & {\hb} & & \\
& {\hs} & & {\hb} & & \\
{\hs} & & {\hs} & & \\
& {\hs} & & {\hb} & & \\
{\hs} & & {\hb} & & \\
& {\hb} & & {\hs} & & \\
{\hb} & & {\hs} & & \\
} \\
\end{tabular}
\end{center}
\caption{Internal columns of a braid cluster}
\label{f:internal_heap}
\end{figure}

Conversely, suppose there exists a fixed a column partition $\{ C_0, B_1, C_1, \ldots, B_k,
C_k \}$ as in the beginning of the proof and we are given either a single heap
$c^{(0)} \in F_{|C_0|}$ if $k=0$ or a sequence of heaps
\begin{equation}\label{e:sequences}
c^{(0)}, b^{(1)}, c^{(1)}, b^{(2)}, \dots, b^{(k)}, c^{(k)}
\end{equation}
where $k > 0$, $c^{(0)} \in L_{|C_0|}$, $c^{(k)} \in R_{|C_k|}$, every other
$c^{(i)} \in M_{|C_i|}$ and each $b^{(i)}$ is a heap fragment on $|B_i|$ columns
with the canonical form shown in Figure~\ref{f:internal_heap}.  Then, we may
form a maximally clustered permutation $w$ with $k$ braid clusters as follows.  If
$k = 0$, take $w$ to be the permutation whose heap is $c^{(0)}$.  Otherwise, for
each $i < k$ we expand the single entry in the rightmost column of $c^{(i)}$ to
a pair of entries.  This has the effect of reversing the $\pi$ map on each
$c^{(i)}$, which we denote by $\pi^{-1}(c^{(i)})$.  Then, glue the columns
together in order to form a single heap, so that the rightmost pair of entries
in each $\pi^{-1}(c^{(i)})$ surrounds the leftmost pair of entries in
$b^{(i+1)}$ if any, or the unique leftmost entry in $\pi^{-1}(c^{(i+1)})$ if
$b^{(i+1)}$ is empty.  Also, each $b^{(i)}$ is glued to $\pi^{-1}(c^{(i)})$ so
that the pair of entries in the rightmost column of $b^{(i)}$ surrounds the
unique entry in the leftmost column of $\pi^{-1}(c^{(i)})$.  We call the
permutation to which this heap corresponds $w$.

Every minimal pair of entries $x$ and $y$ from an internal column of
$\pi^{-1}(c^{(i)})$ must have a distinct resolution because $c^{(i)}$ is
fully commutative.  Hence, we observe that the heap we constructed is reduced
and maximally clustered because it has the form given in
Proposition~\ref{p:canonical_form}.

Next, observe that if $w$ heap-contains a pattern $h \in H$ then the pattern
instance must lie in one of the $c^{(i)}$ To see this, let $x$ and $y$ be a
minimal pair of entries lying in column $i \in [2, n-1]$ of $h$.  Because $h$
is fully commutative, $x$ and $y$ have a distinct resolution, so there exist
entries from columns $i+1$ and $i-1$ lying strictly between $x$ and $y$ in the
heap poset.  However, no column in the heap of any commutativity class for a
braid cluster has this property by the uniqueness statement in
Definition~\ref{d:mc}, since each commutativity class for a braid cluster has a
reduced expression representative that is itself a braid cluster.  Thus, no
heap-instance of $h$ uses the internal columns of a braid cluster.  Since we
are assuming that each $c^{(i)}$ heap-avoids the patterns from $H$, we have
that $w$ heap-avoids the patterns from $H$.

Finally, note that if we compose the constructions we have given above in either
order, then we obtain the identity.  Hence, we have shown a bijection and the
enumerative formulas follow.  $F(x)$ contributes terms of the form where $k =
0$.  Using Lemma~\ref{l:pieces}, we obtain the generating function that counts
sequences of the form given in (\ref{e:sequences}) as the product
\[ L(x) \cdot {1 \over {1-x}} \cdot {1 \over {1 - M(x) \cdot {1 \over {1-x}}}}
\cdot L(x) \]
because each term of the sequence in (\ref{e:sequences}) is independent.
Multiplying the numerator and denominator of the third factor by $(1-x)$, we
obtain the second result.  In the case where $w$ is freely braided, each
$b^{(i)}$ term is $\emptyset$, so the corresponding generating function is
simply
\[ L(x) \cdot {1 \over {1 - M(x)}} \cdot L(x) \]
which gives the first result.
\end{proof}

\begin{corollary}\label{c:enumerate}
Let $H$ be a finite set of connected fully commutative permutations each of
whose heaps contains at least two entries in each internal column.
If $F(x) = \sum_{n \geq 0} |S_{n+1}^{[321]}(H)| \ x^n$
is a rational (respectively, algebraic) generating function, then
\[ \sum_{n \geq 0} |S_{n+1}^{\{[3421], [4231], [4312], [4321]\}}(H)| \
x^n \text{\ \ and \ \ } \sum_{n \geq 0} |S_{n+1}^{\{[3421], [4312], [4321]\}}(H)| \ x^n. \]
are also rational (respectively, algebraic).
\end{corollary}

Using Theorems~\ref{t:translate_olh} and \ref{t:enumerate}, we can enumerate
several interesting classes characterized by permutation pattern avoidance.
Note that we index the coefficients in the generating functions by the rank of
the Coxeter group rather than by the number of entries appearing in the 1-line
notation.

\bigskip
\noindent
\begin{tabular}{|p{0.95in}|l|l|}
\hline
& Generating function & Initial sequence \\
\hline
$[321]$-avoiding & ${{1 - 2x - \sqrt{1-4x}} \over 2x^2}$ &
{\tiny $1 + 2
x + 5 x^2 + 14 x^3 + 42 x^4 + 132 x^5 + 429 x^6 + 1430 x^7 + \dots$ }  \\
\hline
$L(x)$ \ \ \ \ \ \ for $H = \emptyset$ & ${{1 - 3x - (x-1) \sqrt{1-4x}} \over 2x^2}$ & {\tiny
$x + 3 x^2 + 9 x^3 + 28 x^4 + 90 x^5 + 297 x^6 + 1001 x^7 + \dots$ } \\
\hline
$M(x)$ \ \ \ for $H = \emptyset$ & ${{1 - 4x + 3x^2 -2x^3 - (x-1)^2 \sqrt{1-4x}} \over 2x^2}$
& {\tiny $2 x^2 + 6 x^3 + 19 x^4 + 62 x^5 + 207 x^6 + 704 x^7 + \dots$ } \\
\hline
Freely-braided & ${{2x-2x^2-2x \sqrt{1-4x}} \over {-1+4x-x^2+2x^3+(x-1)^2
\sqrt{1-4x}}} $ & {\tiny $1 + 2x + 6x^2 + 20x^3 + 71 x^4 + 260 x^5 +
971x^6 + 3674x^7 + \dots$ } \\
\hline
Maximally-clustered & ${{2x} \over {-1 + 4x - 2x^2 + \sqrt{1-4x}}}$ &
{\tiny $1+2x+6x^2+21x^3+78x^4+298x^5+1157x^6+4539x^7+\dots$ } \\
\hline
$[321]$-hexagon avoiding & ${{-x^5+x^4+3x^3-4x^2+4x-1} \over
{x^6-4x^5-4x^4+9x^3-11x^2+6x-1}}$ & {\tiny $1 + 2 x + 5 x^2 + 14 x^3 + 42
x^4 + 132 x^5 + 429 x^6 + 1426 x^7 + \dots$ } \\
\hline
$L(x)$ for $H = \{[46718235]\}$ & ${{2x^5+2x^4-2x^3+3x^2-x} \over
{x^6-4x^5-4x^4+9x^3-11x^2+6x-1}}$ & \tiny { $x + 3 x^2 + 9 x^3 + 28 x^4 +
90 x^5 + 297 x^6 + 997 x^7 + \dots$ } \\
\hline
$M(x)$ for $H = \{[46718235]\}$ & ${{-x^7+2x^6+4x^5-5x^4+6x^3-2x^2} \over
{x^6-4x^5-4x^4+9x^3-11x^2+6x-1}}$ & {\tiny $2 x^2 + 6 x^3 + 19 x^4 + 62
x^5 + 207 x^6 + 700x^7 + \dots$ } \\
\hline
Freely-braided hexagon-avoiding & ${{-x^6-2x^5+2x^4+x^3-3x^2+4x-1} \over
{x^7-x^6-8x^5+x^4+3x^3-9x^2+6x-1}}$ & {\tiny $1 + 2x + 6x^2 + 20x^3 + 71
x^4 + 260 x^5 + 971x^6 + 3670x^7 + \dots$ } \\
\hline
Maximally-clustered hexagon-avoiding & ${{3x^5+x^4-5x^3+7x^2-5x+1} \over
{-3x^6+4x^5+8x^4-14x^3+15x^2-7x+1}}$ & {\tiny
$1+2x+6x^2+21x^3+78x^4+298x^5+1157x^6+4535x^7+\dots$ } \\
\hline
\end{tabular}
\bigskip

\begin{remark}
The first line is the Catalan generating function which appears in this context
by \cite{simion-schmidt}, while the sixth line is Theorem~\ref{t:s-w} due to
\cite{s-w}.  The freely braided permutations have previously been enumerated in
\cite{mansour}, while the other formulas seem to be new.  They follow from
Theorem~\ref{t:translate_olh} together with Lemma~\ref{l:pieces} and
Theorem~\ref{t:enumerate} by taking $H = \emptyset$ and $H = \{ [46718235] \}$,
respectively.
\end{remark}

\begin{corollary}
The number $b_n$ of freely-braided hexagon-avoiding permutations in $S_n$ satisfies the
recurrence
\[ b_{n+1} = 6 b_{n} - 9b_{n-1} + 3b_{n-2} + b_{n-3} - 8b_{n-4} - b_{n-5} +
b_{n-6} \]
and the number $m_n$ of maximally-clustered hexagon-avoiding permutations in $S_n$ satisfies the
recurrence
\[ m_{n+1} = 7 m_{n} - 15m_{n-1} + 14m_{n-2} - 8m_{n-3} - 4m_{n-4} + 3m_{n-5} \]
for all $n \geq 9$, with initial conditions given by Figure~\ref{f:mc.enum}.
\end{corollary}


\bigskip
\section{Diamond reductions}\label{s:diamond}

In this section, we restrict to considering fully commutative permutations.

\begin{definition}
Suppose that $h$ is a connected, fully commutative permutation whose heap
contains at least two entries in each internal column.  We say that every
minimal pair of entries together with the entries of their distinct resolution
forms a \em minimal diamond \em inside the heap of $h$.  Form a new heap whose
entries correspond to minimal diamonds.  The poset structure of the new heap is
inherited from the poset structure on the minimal diamonds in the heap of $h$ by
taking the transitive closure of the relation that two minimal diamonds are
related if they share an edge.  The heap obtained in this way from $h$ is called
the \em diamond reduction \em of $h$.
\end{definition}

Pictorially, the heap of the diamond reduction of $h$ is obtained from the heap
of $h$ by adding entries at the centers of all minimal diamonds and then erasing
the heap of $h$.

\begin{example}\label{ex:3hex}
The diamond reduction of the hexagon $s_5 s_6 s_7 s_3 s_4 s_5 s_6 s_2 s_3 s_4
s_5 s_1 s_2 s_3$ is the \em 3-hexagon \em $s_4 s_5 s_2 s_3 s_4 s_1 s_2$.
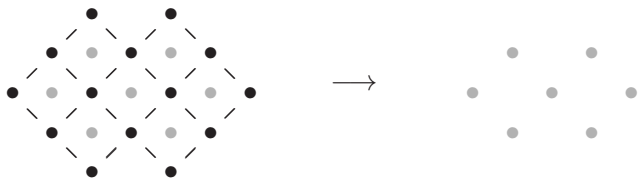
\[
\xymatrix @=3pt @! {
 & & \hf \ar@{-}[dl] \ar@{-}[dr] & & \hf \ar@{-}[dl] \ar@{-}[dr] \\
 & \hf \ar@{-}[dr] \ar@{-}[dl] & \hb & \hf \ar@{-}[dl] \ar@{-}[dr] & \hb & \hf \ar@{-}[dl] \ar@{-}[dr]\\
 \hf \ar@{-}[dr] & \hb & \hf \ar@{-}[dr] \ar@{-}[dl] & \hb & \hf \ar@{-}[dl] \ar@{-}[dr] & \hb & \hf \ar@{-}[dl] \\
 & \hf \ar@{-}[dr] & \hb & \hf \ar@{-}[dl] \ar@{-}[dr] & \hb & \hf \ar@{-}[dl] \\
 & & \hf & & \hf \\
}
\parbox[t]{1in}{ \vspace{.3in}
\hspace{.3in} $\longrightarrow$ }
\xymatrix @=3pt @! {
 & & & \\
 & \hb & & \hb \\
 \hb & & \hb & & \hb \\
 & \hb & & \hb \\
}
\]
\end{example}

\begin{proposition}\label{p:dr_bij}
The diamond reduction is a bijection from the set of connected fully commutative
heaps on $n$ columns with at least two entries in each internal column to the
set of connected fully commutative heaps on $n-2$ columns.
\end{proposition}
\begin{proof}
Let $h$ be a connected fully commutative heap on $n$ columns with at least two
entries in each internal column.  The diamond reduction $g$ of $h$ gives a
reduced heap because we can represent the minimal diamonds by their maximal
entries (rather than the centers) and the result is a convex subposet of the
heap of $h$.  The heap of $g$ is also connected since every internal column of
the heap of $h$ has a minimal diamond.  The diamond reduction can be reversed by
forming a heap with minimal diamonds centered at the heap entries of $g$.
Hence, the diamond reduction of $h$ is fully commutative because any short-braid
instance in the heap of $g$ would imply the existence of a short-braid in $h$
when we consider the minimal diamonds centered at the entries of the heap of
$g$.  We also reduce the number of columns by 2 since only the internal columns
of the heap of $h$ support minimal diamonds that become entries in the heap of
$g$.
\end{proof}

Our main goal in this section is to describe how the generating functions for
$\{|S_{n+1}^{[321]}(h)|\}$ and $\{|S_{n+1}^{[321]}(g)|\}$ are related when $g$
is the diamond reduction of $h$.

\begin{lemma}\label{l:dr_connected}
Let $H$ be a set of connected, fully commutative permutations.
Suppose
\[ F(x) = \sum_{n \geq 0} |S_{n+1}^{[321]}(H)| \ x^n \text{\ \  and \ \ }
F_{c}(x) = \sum_{n \geq 0} |\{ p \in S_{n+1}^{[321]}(H) : \text{ $p$ is connected } \}| \ x^n, \]
with $F(0) = 1$ and $F_{c}(0) = 0$.
Then the generating functions are rationally related according to
\[ F(x) = {{1+F_{c}(x)} \over {1-x-x F_{c}(x)}} \text{\ \  and \ \ }
F_{c}(x) = { {F(x)-x F(x)-1} \over {1+x F(x)} }. \]
\end{lemma}
\begin{proof}
To see this, observe that every not-necessarily-connected $H$-avoiding heap on
$n$ columns decomposes uniquely into connected components, each of which is
counted by $F_{c}(x)$.  Conversely, if we have an ordered collection
$\mathcal{C}$ of connected permutations that all heap-avoid the patterns from
$H$, then we can place them together from left to right, to obtain a heap which
has the heaps from $\mathcal{C}$ as its connected components.

The generating function identity that we obtain from this observation is
\[ F(x) = {1 \over {1-x}} + {1 \over {1-x}} \cdot F_{c}(x) \cdot {1 \over {1 - {
{x \over {1-x}} F_{c}(x) }}} {1 \over {1-x}}. \]
Specifically, we either have some connected components, or none.  The initial
${1 \over {1-x}}$ term corresponds to the case where we have none.  Otherwise,
we have zero or more initial empty columns, followed by at least one connected
component, followed by zero or more copies of empty columns together with other
connected components, finally followed by zero or more empty columns on the
right side of the heap.  This is the second term.  This formula simplifies to
the one given in the statement and can be inverted.
\end{proof}

One fundamental subclass of the fully commutative permutations are those with no
minimal diamonds at all.  It is straightforward to see that this corresponds to
having at most one entry in each column, or equivalently to heap-avoiding $s_2
s_1 s_3 s_2$.  This pattern class was previously enumerated in \cite{fan1} and
\cite{west2}.  Also, \cite{t3} enumerates these permutations using the statistic of
Coxeter length.

\begin{lemma}\label{l:diamond_avoiding}
We have
\[ \sum_{n \geq 0} |S_{n+1}^{[321], [3412]}| \ x^n = \sum_{n \geq 0} |S_{n+1}^{[321]}(s_2 s_1 s_3 s_2)| \ x^n = { {1-x} \over {1 - 3x + x^2 }}. \]
\end{lemma}
\begin{proof}
The second equality follows from Theorem~\ref{t:translate_olh}.  Let $w$ be a
connected element of $S_{n+1}^{[321]}(s_2 s_1 s_3 s_2)$.  Then the heap of $w$
is a lattice path consisting of $n-1$ steps that are either ``up'' or ``down.''
As each step is independent, the generating function for these is ${{x} \over
{1-2x}}$.  Applying Lemma~\ref{l:dr_connected} gives the result.
\end{proof}

We are now in a position to state and prove our main result in this section.

\begin{theorem}\label{t:main_diamond}
Suppose that $h$ is a connected, fully commutative permutation whose heap
contains at least two entries in each internal column and let $g$ be the diamond
reduction of $h$.  If
\[ G(x) = \sum_{n \geq 0} |S_{n+1}^{[321]}(g)| \ x^n \text{, \ \  \ \ }
G_{c}(x) = \sum_{n \geq 0} |\{ p \in S_{n+1}^{[321]}(g) : \text{ $p$ is
connected } \}| \ x^n, \]
and $F(x) = \sum_{n \geq 0} |S_{n+1}^{[321]}(h)| \ x^n$,
with $G(0) = 1 = F(0)$ and $G_{c}(0) = 0$,
then the generating functions are rationally related according to
\begin{equation}\label{e:dr_main}
F(x) = { {1-x-x G_{c}(x)} \over {1-3x+x^2+(x^2-x)G_{c}(x)}}
= { {1} \over {1-2x-x^2 G(x)}}.
\end{equation}
\end{theorem}
\begin{proof}
Suppose $w \in S_{n+1}^{[321]}(h)$.  We assign each of the columns $\{1, \ldots n\}$
of the heap of $w$ to intervals $E_0, D_1, E_1, \ldots, D_k, E_k$ according to
whether the column supports more than one entry.  The intervals $D_i$ are
defined to be precisely those that support maximal connected fully commutative
subheaps of the heap of $w$ such that the internal columns of $D_i$ each support
at least two entries.  The remaining intervals $E_i$ consist of columns that
support at most one entry.  In particular, the $E_i$ may include empty columns.

For example, suppose
\[ w = s_1 s_2 s_{12} s_8 s_9 s_{10} s_{11} s_5 s_6 s_7 s_8 s_4 s_5 s_6. \]
Then the heap of $w$ is drawn below, together with its column assignments.
{\scriptsize
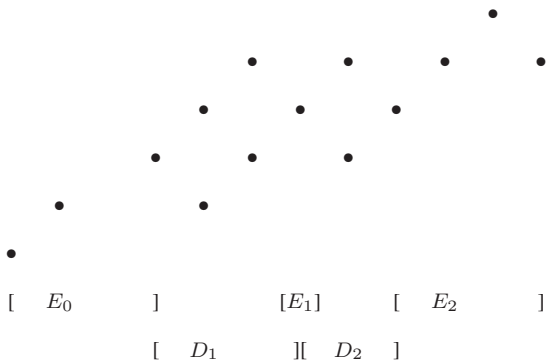
\[ \xymatrix @=-4pt @! {
& & \\
& {\hs} & & {\hs} & & {\hs} & & {\hs} & & {\hs} & & {\hf} \\
{\hs} & & {\hs} & & {\hs} & & {\hf} & & {\hf} & & {\hf} & & {\hf} \\
& {\hs} & & {\hs} & & {\hf} & & {\hf} & & {\hf} & \\
{\hs} & & {\hs} & & {\hf} & & {\hf} & & {\hf} & \\
& {\hs} & {\hf} & {\hs} & & {\hf} & & {\hs} & & {\hs} \\
{\hs} & {\hf} & {\hs} & & {\hs} & & {\hs} & & {\hs} & \\
& \text{[} & E_0 & & \text{]} &     & &\text{[$E_1$]}   &     & \text{[} & E_2 & & \text{]} \\
&             &  & & \text{[} & D_1 & &\text{][} & D_2 & \text{]} &  & & \\
} \] }

Let $w|_{D_i}$ denote the restriction of the heap of $w$ to the columns in the
interval $D_i$.  By the hypotheses given, each $w|_{D_i}$ must heap-avoid $h$.
Moreover, we can form the diamond reduction of $w|_{D_i}$ whose heap must be
connected, heap-avoids $g$, and has 2 fewer columns than $w|_{D_i}$.  Each
$w|_{E_i}$ must heap-avoid $s_2 s_1 s_3 s_2$ and contains entries in the
extremal columns of $E_i$ that are shared with $D_{i}$ or $D_{i+1}$.

Conversely, suppose there exists a fixed column assignment $E_0, D_1, E_1,
\ldots, D_k, E_k$  and we are given a sequence of heaps $e_0, d_1, e_1, \ldots,
d_k, e_k$ such that $e_0 \in S_{|E_0|+1}^{[321]}(s_2 s_1 s_3 s_2)$ with one
entry in the rightmost column of $E_0$, $d_i \in S_{|D_i|-1}^{[321]}(g)$ and
$d_i$ is connected for $1 \leq i \leq k$, $e_i \in S_{|E_i|+1}^{[321]}(s_2 s_1
s_3 s_2)$ for $1 \leq i \leq k-1$ with one entry in each of the extremal columns
of $E_i$ and $e_k \in S_{|E_i|+1}^{[321]}(s_2 s_1 s_3 s_2)$ with one entry in
the leftmost column of $E_k$.  Then we can apply Proposition~\ref{p:dr_bij} to
form the reverse diamond reduction $\widetilde{d_i}$ of $d_i$ and glue these
heaps together to obtain an element of $S_{n+1}^{[321]}(h)$.  Specifically, we
identify the unique entry in the rightmost column of $e_i$ with the unique entry
in the leftmost column of $\widetilde{d_{i+1}}$ for each $0 \leq i \leq k-1$,
and we also identify the unique entry in the leftmost column of $e_i$ with the
unique entry in the rightmost column of $\widetilde{d_{i}}$ for each $1 \leq i
\leq k$.

If we compose the constructions we have given in either order, we obtain the
identity.  Hence, we have shown a bijection.

Let $E(x)$ be the generating function appearing in
Lemma~\ref{l:diamond_avoiding}.  Define $E_{LR}(x) = E(x) - x E(x) - 1$ and
$E_{M}(x) = E(x) - 2x E(x) + x^2 E(x) - 1 + x$.  It is straightforward to verify
using inclusion-exclusion that $E_{LR}(x)$ counts the number of $s_2 s_1 s_3
s_2$-avoiding heaps with one entry in at least one of the two extremal columns,
and $E_{M}(x)$ counts the number of $s_2 s_1 s_3 s_2$-avoiding heaps with one
entry in both of the extremal columns.

Putting all of these observations together, we have the generating function
identity
\begin{equation}\label{e:main_diamond}
F(x) = E(x) + E_{LR}(x) {1 \over {1 - G_{c}(x) E_{M}(x)}} G_{c}(x) E_{LR}(x).
\end{equation}
Here, $E(x)$ contributes terms of the form where $k = 0$.  This formula
simplifies to the first equality given in the statement.  The second equality
follows from Lemma~\ref{l:dr_connected}.
\end{proof}

\begin{remark}
In the nicest cases, this result allows us to reduce the problem of enumerating
permutations that avoid a given heap pattern to counting lattice paths that
avoid a certain consecutive subpath.  This problem is well known to have a
rational generating function and the transfer matrix method can be used to find
the generating function explicitly.  See \cite[Example 4.7.5]{ec1} for details.
\end{remark}

\begin{example}
\[ \sum_{n \geq 0} |S_{n+1}^{[321], [356124], [456123]}| \ x^n = \sum_{n \geq 0} |S_{n+1}^{[321]}(s_4 s_5 s_2 s_3 s_4 s_1 s_2 s_3)| \ x^n = { {1-3x+2x^2-x^3} \over {1 - 5x + 7x^2 -4x^3+x^4}}. \]

The first equality follows from Theorem~\ref{t:translate_olh}.
The heap pattern $p = s_4 s_5 s_2 s_3 s_4 s_1 s_2 s_3$ has a diamond reduction to $q = s_1 s_3 s_2$.
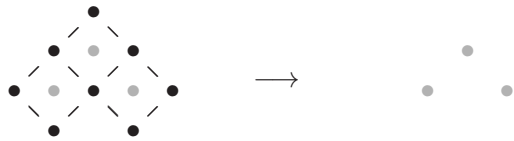
\[
\xymatrix @=3pt @! {
 & & \hf \ar@{-}[dl] \ar@{-}[dr] & & \\
 & \hf \ar@{-}[dl] \ar@{-}[dr] & \hb & \hf \ar@{-}[dl] \ar@{-}[dr] \\
 \hf \ar@{-}[dr] & \hb & \hf \ar@{-}[dl] \ar@{-}[dr] & \hb & \hf \ar@{-}[dl] \\
 & \hf & & \hf \\
}
\parbox[t]{1in}{ \vspace{.3in}
\hspace{.3in} $\longrightarrow$ }
\xymatrix @=3pt @! {
 & & & \\
 \hs & & \hb & & \hs \\
 & \hb & & \hb \\
}
\]
The number of connected permutations that heap-avoid $q$ is the same as the
number of lattice paths that never contain a consecutive up-down subpath.  It is
straightforward to see that there are $n$ such lattice paths that use $n$ nodes.
The corresponding generating function is $G_{c}(x) = {x \over {(1-x)^2}}$ and
substituting this into Equation~(\ref{e:dr_main}) yields the result.
\end{example}

\begin{example}
\[ \sum_{n \geq 0} |S_{n+1}^{[321], [46718235], [46781235], [56718234], [56781234]}| \ x^n
= \sum_{n \geq 0} |S_{n+1}^{[321]}(s_5 s_6 s_7 s_3 s_4 s_5 s_6 s_2 s_3 s_4 s_5 s_1 s_2 s_3)| \ x^n \]
\[ = { {1-4x+4x^2-3x^3-x^4+x^5} \over {1 - 6x + 11x^2 -9x^3+4x^4+4x^5-x^6}}. \]
recovering Theorem~\ref{t:s-w}

The first equality follows from Theorem~\ref{t:translate_olh}.
The hexagon pattern has a diamond reduction to $q = s_4 s_5 s_2 s_3 s_4 s_1 s_2$
as shown in Example~\ref{ex:3hex}.  However, the diamond reduction of $q$ is
$s_1 s_3$ which is disconnected.  Hence, we must enumerate the permutations
heap-avoiding $q$ directly.  We use the same column assignment as in the proof
of Theorem~\ref{t:main_diamond}.  In order to heap-avoid $q$, the intervals of
columns supporting diamonds must restrict to fully commutative permutations
whose diamond reduction is a connected monotonic lattice path.  Also, we can no
longer glue two diamond-containing regions together along a trivial lattice path
with one column.  Hence, we must modify the $E_{M}(x)$ term of
Equation~(\ref{e:main_diamond}) by subtracting $x$.

The monotonic lattice paths are counted by $G_{c}(x) = {{2x} \over {1-x}} -x$.
Substituting this into Equation~(\ref{e:main_diamond}) along with $E_{M}(x) =
E(x) - 2x E(x) + x^2 E(x) - 1$ gives the generating function for all
$q$-heap-avoiding permutations.  Applying Theorem~\ref{t:main_diamond} then
yields the result.
\end{example}


\bigskip
\section*{Acknowledgments}

We thank Sara Billey, Jozsef Losonczy, and Richard Green for many useful
suggestions, as well as Julian West for introducing us to Olivier Guibert's
code \cite{guibert} for permutation pattern enumeration.


\end{document}